\documentclass{amsart}
\usepackage{amsfonts}
\usepackage{amsmath,amssymb}	
\usepackage{amsthm}
\usepackage{amscd}
\usepackage{graphics}
\usepackage{graphicx}

{
\theoremstyle{definition}
\newtheorem{Def}{Definition}
\newtheorem{Ex}{Example}
\newtheorem{Rem}{Remark}

}

\newtheorem{Cor}{Corollary}
\newtheorem{Prop}{Proposition}
\newtheorem{Thm}{Theorem}

\begin{document}
\title[Generalized Reeb spaces of special generic maps and lifts of maps]{Generalizations of Reeb spaces of special generic maps and applications to a problem on lifts of Morse functions}

\author{Naoki Kitazawa}
\address{19-9-606 Takayama, Tsuzuki-ku, Yokohama, Kanagawa 224-0065, JAPAN}
\email{naokikitazawa.formath@gmail.com}
\maketitle
\begin{abstract}
A {\it Reeb space} is defined as the space of all the connected components of inverse images of a smooth map and it is a fundamental tool in the studies
 of differentiable manifolds using generic smooth maps whose codimension is not positive such as Morse functions, their higher dimensional versions including {\it fold maps} and general {\it stable} maps.

 A {\it special generic} map is a fold map and regarded as a generalization of Morse functions with just $2$ singular points on homotopy spheres and the Reeb space
 is a compact manifold whose dimension is equal to that of the target manifold and which can be immersed into the target manifold. 

In this paper, we generalize a quotient map onto a Reeb space of a special generic map. We define a map onto a polyhedron locally a quotient map induced from a special generic
 map. In fact, there have been such generalizations. For example, in 1996, Kobayashi and Saeki generalized stable maps into
 the plane of closed manifolds of dimensionl larger than $3$ to study geometry of such maps systematically and the author recently defined such tools. 

 We take advantage of the generalized maps to construct {\it lifts} of Morse functions of a certain class; the composition of the lift and the canonical projection is the original function. It is an
 answer of an explicit problem in the studies of lifts or desingularizations of smooth maps, or maps such that the compositions of the found maps and the canonical projections are original maps, which are fundamental and important in the studies of smooth maps and applications to algebraic and differential topology of manifolds.   
 
\end{abstract}

\section{Introduction.}
\label{sec:1}
\subsection{Fundamental tools.}
A {\it Reeb space} is defined as the quotient space of all the connected components of inverse images of a smooth map whose codimension is minus.
It is a fundamental important tool in the studies of Morse functions, {\it fold maps}, which are higher dimensional versions of Morse functions and general generic maps including {\it stable}
 maps and applications to geometry of differentaible manifolds. For Reeb spaces, see \cite{reeb} for example.

 As a fundamental fact, the Reeb space of a Morse functions is a graph and generally, for maps mentioned here, the Reeb spaces
 are polyhedra respecting the canonical triangulations of manifolds and the maps. For such facts, see \cite{shiota} and see also \cite{hiratukasaeki} and \cite{hiratukasaeki2} for example. In addition, for Morse functions, fold maps and
 stable maps etc., see \cite{golubitskyguillemin} for example.

We review {\it fold maps} and {\it special generic maps}.

A  {\it fold map} is a smooth map such that each singular point $p$ is
 of the form $$(x_1,\cdots,x_m) \mapsto (x_1,\cdots,x_{n-1},\sum_{k=n}^{m-i(p)}{x_k}^2-\sum_{k=m-i(p)+1}^{m}{x_k}^2)$$ for some
     integers $m,n,i(p)$ satisfying $m \geq n \geq 1$ and $0 \leq i(p) \leq \frac{m-n+1}{2}$. Note that a Morse function is a fold map.
The integer $i(p)$ is taken as a non-negative integer not larger than $\frac{m-n+1}{2}$ uniquely and we call $i(p)$ the {\it index} of $p$. The
 set of all the singular points of an index is a smooth submanifold of dimension $n-1$ and the
     restriction of the fold map to the singular set is an immersion and it is transversal if the fold map is {\it stable} (stable Morse functions exist densely on smooth closed manifolds
 and fold maps are generically stable).  

A {\it special generic map} is a fold map such that for all the singular points, the indices are $0$. Morse functions with just $2$ singular points on homotopy spheres and the canonical projections of unit spheres are simplest examples of special generic maps. We can know fundamental properties and important works on special generic maps in \cite{saeki}, \cite{saekisakuma} and \cite{wrazidlo} for example. 

\subsection{Contents of this paper.}
The Reeb space of a special generic map is regarded as an immersed compact manifold with a non-empty boundary whose dimension is same as that of the target manifold. In this paper, we generalize the quotient maps onto the Reeb spaces induced from special generic maps into Euclidean spaces as smooth maps onto compact smooth manifolds and give applications to a problem questioning that there exists a {\it lift} of a smooth map. More precisely, for a smooth map $f$ from an $m$-dimensional manifold into an $n$-dimensional manifold, does there exists
 an immersion, embedding or a map of an appropriate class $f_0$ or a {\it lift} into ${\mathbb{R}}^{n+k}$ such that for the canonical projection ${\pi}_{n+k,n}$,  $f={\pi}_{n+k,n} \circ f_0$ holds?   

Note that such generalizations of maps have been done in several works. For exmaple,  in 1996, Kobayashi and Saeki generalized stable maps into
 the plane of closed manifolds of dimensionl larger than $3$ to study geometry of such maps systematically and the author recently defined generalized maps of quotient maps induced from proper stable
 maps and more generally, proper triangulable smooth maps to show a theorem on the homology groups of Reeb spaces for more general spaces. As another example, a {\it shadow} \cite{turaev}
 is regarded as an extension of the quotient maps onto the Reeb spaces induced from stable fold maps on closed $3$-dimensional manifolds into the plane satisfying a condition on inverse images containing singular points and related studies are performed in \cite{costantinothurston} and \cite{ishikawakoda}.  

Last, studies of lifts of smooth maps are also fundametal and important in the studies of smooth maps and applications to algebraic and differential topology of differentaible manifolds. As a familiar example, plane curves can be lifted to classical knots in the $3$-dimensional space and there have been other various
results on the problems.

\begin{itemize}
\item Morse functions and stable
 maps of $2$ or $3$-dimensional manifolds into the plane are liftable to immersions or embeddings into appropriate dimensional spaces as shown in \cite{haefliger}, \cite{levine}, \cite{yamamoto} and \cite{yamamoto2} for example.
\item Generic maps between equi-dimensional manifolds are lifted to immersions or embeddings into one-dimensional higher spaces as done in \cite{saito} and see also \cite{blankcurley} for example.
\item Various special generic maps are lifted to immersions or embeddings as done in \cite{saekitakase} and \cite{nishioka} for example.
\item As a recent work \cite{kitazawa3}, the author considered a problem questioning that a {\it normal spherical Morse function} can be lifted to a special generic map into an Euclidean space show several results.
\end{itemize}

Here, we define a {\it normal spherical} Morse function, which is explained in \cite{kitazawa3}, and see also \cite{saekisuzuoka} for example. 
\begin{Def}
\label{def::1}
A stable fold map from a closed manifold of dimension $m$ into ${\mathbb{R}}^n$ satisfying $m \geq n$ is said to be {\it normal spherical} ({\it standard-spherical}) if the inverse image
 of each regular value is a disjoint union of (resp. standard) spheres (or points) and 
the connected component containing a singular point of the inverse image of a
 small interval intersecting with the singular value set at once in its interior is either of the following.
\begin{enumerate}
\item The ($m-n+1$)-dimensional standard closed disc.
\item A manifold PL homeomorphic to an ($m-n+1$)-dimensional compact manifold obtained by removing the interior of three disjoint ($m-n+1$)-dimensional smoothly embedded closed discs from the ($m-n+1$)-dimensional standard sphere.
\end{enumerate}  
\end{Def}

We review a fundamental property.

\begin{Prop}[\cite{saekisuzuoka},\cite{kitazawa},\cite{kitazawa2}]
\label{prop:1}
For a normal spherical Morse function $f$ on a closed manifold $M$ of dimension $m>1$, the quotient map $q_f$ on $M$ to the Reeb space $W_f$ induces an isomorphism of homology, cohomology and homotopy groups whose degrees are smaller than $m-1$.

\end{Prop}

In this paper, we study a problem related to the last one of the five kinds of results above and show an advanced result by make use of introduced generalized maps. The contents of the present paper is as the following.
In the next section, we generalize the quotient maps to the Reeb spaces to smooth maps on closed manifolds onto compact smooth manifolds whose dimensions are lower than those of the source manifolds. 
In the last section, as a main work, we consider a spherical Morse function and represent this as a composition of two maps belonging to appropriate classes, which is regarded as a new answer for the problem studied in \cite{kitazawa3} and other related problems on lifts of smooth maps. 

\section{Pseudo special generic maps.}
\begin{Def}
\label{def:2}
Let $m>n$ be positive integers.
A smooth map $f_p$ from a closed manifold $M$ of dimension $m$ onto a compact manifold $W_p$ of dimension $n$ satisfying $\partial W_p \neq \emptyset$ is
 said to be {\it pseudo special generic} if the following hold.
\begin{enumerate}
\item $f_p {\mid}_{{f_p}^{-1}(W_p-\partial W_p)}:{f_p}^{-1}(W_p-\partial W_p) \rightarrow W_p-\partial W_p$ gives a smooth $S^{m-n}$-bundle.
\item For a small collar neighborhood $N(\partial W_p)$ of $\partial W_p$ in $W_p$, regarded as a trivial bundle $\partial W_p \times [0,1]$ where $\partial W_p \times \{0\}$
 corresponds to the boundary $\partial W_p$, for each point $(p,0) \in \partial W_p \times \{0\}$ and a small open neighborhood $U_p$,
 $f {\mid}_{{f_p}^{-1}(U_p \times [0,1])}:f^{-1}(U_p \times [0,1]) \rightarrow U_p \times [-1,1] $ has the same local form as a singular point of index $0$ of a fold map
 from an $m$-dimensional manifold into ${\mathbb{R}}^n$.
 From fundamental discussion of \cite{saeki} for example, ${f_p}^{-1}(N(\partial W_p))$ is a linear $D^{m-n+1}$-bundle over $\partial W_p$ given by the composition $f_p$ and the
 canonical projection and the bundle is seen as a normal bundle of the submanifold $f^{-1}(\partial W_p) \subset M$.
\end{enumerate}
\end{Def}
All the quotient maps onto the Reeb spaces defined from special generic maps into Euclidean spaces are regarded as pseudo special generic.
\begin{Def}
\label{def:3}
A pseudo special generic map is said to be {\it trivial} if the bundles appearing in the two conditions of Definition \ref{def:2} are trivial.
\end{Def}
\begin{Ex}
\begin{enumerate}
\item
The Reeb space of a special generic map of a homotopy sphere into an Euclidean space is a contractible smooth manifold whose dimension is same as that of the target Euclidean space
 and the boundary is a homology sphere with coefficient ring $\mathbb{Z}$.
The quotient map induced from a special generic map of a homotopy sphere into an Euclidean space of dimension smaller than $4$ is trivial. See \cite{saeki}.
\item
Saeki, Takase \cite{saekitakase} and Nishioka \cite{nishioka} have shown that special generic maps satisfying appropriate differential topological conditions admit lifts which are immersions or
 embeddings and most of them are trivial. In addition, the author \cite{kitazawa} has given lifts of spherical Morse functions as special generic maps and seen that some of the resulting special generic maps are trivial
 and that some are not. 
\end{enumerate}
\end{Ex}
\section{Results.}

\begin{Thm}
\label{thm:1}
\begin{enumerate}
\item
For a normal spherical Morse function $f:M \rightarrow \mathbb{R}$ on a closed manifold $M$ of dimension $m>2$ satisfying $m \neq 5$, there exist a compact smooth manifold $W_p$ of dimension $2$, a trivial pseudo special generic map $f_p:M \rightarrow W_p$ and a smooth map $g$ satisfying $f=g \circ f_p$. 
\item Let $m>n$ be integers satisfying $m \neq 5$ and $n=3,4$ and let $m<6$ or $m \geq 6$ and the Gromoll filtration number of all the diffeomorphism on $D^m$ fixing all the boundary points is always larger than $n-1$ {\rm (}for example $(m,n)=(13,3),(13,4)$; see \cite{crowleyschick} and also \cite{kitazawa}{\rm )}.
For a normal standard-spherical Morse function $f:M \rightarrow \mathbb{R}$ on a closed manifold $M$ of dimension $m$, there exist a compact smooth manifold $W_p$ of dimension $n$, a trivial pseudo special generic map $f_p:M \rightarrow W_p$ and a smooth map $g$ satisfying $f=g \circ f_p$.
\end{enumerate}
\end{Thm}
For the proof, we need facts and technique precisely presented in \cite{cerf}, \cite{saeki}, \cite{wrazidlo} and \cite{kitazawa} and we review some of them.

\begin{Def}
\label{def:5}
Let $k$ be an integer larger than $6$. Let $l$ be a positive integer not larger than $k$ and let a diffeomorphism $\phi$ on the unit disc $D^{k-1} \subset {\mathbb{R}}^{k-1}$ fixing all the points in the boundary can be smoothly isotoped to ${\phi}_0$
so that ${{\pi}_{k-1,l-1}} \mid_{D^{k-1}} \circ {\phi}_0={\pi}_{k-1,l-1} {\mid}_{D^{k-1}}$. The {\it Gromoll filtration number} of the diffeomorphism $\phi$ is defined as the maximal number $l$.  
\end{Def}

\begin{Prop}[\cite{cerf}, \cite{wrazidlo}]
\label{prop:2}
\begin{enumerate}
\item
\label{prop:2.1}
 Let $k \geq 6$ be an integer.
We can obtain an orientation preserving diffeomorphism on $S^k$ by regarding $D^k$ as the hemisphere of the unit sphere of ${\mathbb{R}}^{k+1}$ by an injection $x \mapsto (x,\sqrt{1-{\parallel x \parallel}^2}) \in S^k$ and extending
 the diffeomorphism fixing all the points of the boundary on $D^k$ to the unit sphere $S^k$ by using the identity map. If 
the Gromoll filtration number is larger than $l$, then we can smoothly isotope the diffeomorphism to ${\phi}_0$ so that ${\pi}_{k+1,l} {\mid}_{S^k} \circ {\phi}_0 = {\pi}_{k+1,l} {\mid}_{S^k}$
holds.
\item
\label{prop:2.2}
 For a positive integer $k$ smaller than $6$ and not $4$ and a positive integer $l \leq k$, then we can smoothly isotope an orientation diffeomorphism on $S^k$ to ${\phi}_0$ so
 that ${\pi}_{k+1,l} {\mid}_{S^k} \circ {\phi}_0 = {\pi}_{k+1,l} {\mid}_{S^k}$ holds.
\item
\label{prop:2.3}
 Let $k \geq 6$ be an integer. Then the {\it Gromoll filtration number} of the diffeomorphism $\phi$ on the
 unit disc $D^{k-1} \subset {\mathbb{R}}^{k-1}$ fixing all the points in the boundary is larger than $1$.
\item
\label{prop:2.4}
 Let $k$ be an integer larger than $6$, $\Sigma$ be a homotopy sphere of dimension $k$ and $f_{\Sigma}$ be a Morse function with just two singular points.
 then we can smoothly isotope any orientation preserving diffeomorphism on $\Sigma$ to ${\phi}_0$ so
 that $f_{\Sigma} \circ {\phi}_0 = f_{\Sigma}$ holds.
\end{enumerate}
\end{Prop}

We easily obtain the following, which is a fundamental proposition in the proof.

\begin{Prop}
\label{prop:3}
\begin{enumerate}
\item
\label{prop:3.1}
 Let $k \geq 6$ be an integer and $l \leq k$. be a positive integer, then we have
 orientation reversing diffeomorphisms $r_{S^k}$ and $r_l$ on $S^k$ and ${\mathbb{R}}^l$, respectively so that ${\pi}_{k+1,l} {\mid}_{S^k} \circ r_{S^k} =r_l \circ {\pi}_{k+1,l} {\mid}_{S^k}$
holds.
\item
\label{prop:3.2}
 For a positive integer $k$ smaller than $6$ and not $4$ and a positive integer $l \leq k$, 
then we have orientation reversing diffeomorphisms $r_{S^k}$ and $r_l$ on $S^k$ and ${\mathbb{R}}^l$, respectively so that ${\pi}_{k+1,l} {\mid}_{S^k} \circ r_{S^k} =r_l \circ {\pi}_{k+1,l} {\mid}_{S^k}$ holds.
\item
\label{prop:3.3}
 Let $k$ be an integer larger than $6$ and $\Sigma$ be a homotopy sphere of dimension $k$ on which there exists an orientation reversing diffeomorphism and $f_{\Sigma}$ be a Morse function with just two singular points.
 then we have orientation reversing diffeomorphisms $r_{\Sigma}$ and $r$ on $S^k$ and ${\mathbb{R}}$, respectively and $f_{\Sigma} \circ r_{\Sigma} =r \circ f_{\Sigma}$ holds.
\end{enumerate}
\end{Prop}
\begin{proof}[Sketch of the proof]
We are enough to consider two appropriate antipodal poles and send each pole to another pole in the case where the source manifold is a unit sphere or each singular point to
 another singular point in the case where the source manifold is a general homotopy sphere with a Morse function with just two singular points. This gives us a desired orientation
 reversing diffeomorphism on the source homotopy sphere. A desired orientation reversing diffeomorphism on the target space is given by a reflection by a hyperplane containing the origin. See
 also FIGURE \ref{fig:1}.
\end{proof}
\begin{figure}
\includegraphics[width=35mm]{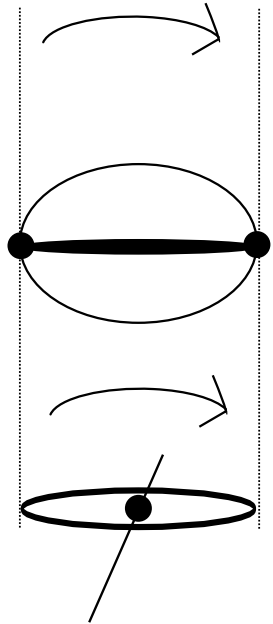}
\caption{Orientation reversing diffeomorphisms of a homotopy sphere and the image by a special generic map (the case of a canonical projection of a unit sphere; the arrows represent each orientation reversing diffeomorphism).}
\label{fig:1}
\end{figure} 
\begin{proof}[Proof of Theorem \ref{thm:1}]
We prove the first part.
First, as done in \cite{kitazawa} for example, we locally lift Morse functions to smooth maps whose singular points are of the same form as that of a singular point of index $0$ of a fold map. 
First, for the connected component of the inverse image of a closed interval
 containing just one singular point, we lift as drawn in FIGURE \ref{fig:2}. 
\begin{figure}
\includegraphics[width=35mm]{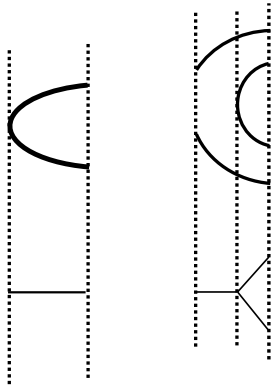}
\caption{Local lifts of normal spherical Morse functions.}
\label{fig:2}
\end{figure} 
For a connected component containing a singular point of index $0$, we lift as the left figure. The image is a $2$-dimensional closed disc with $2$ points in its corner. This is also regarded as a cobordism of special generic functions such that one of the function is null and the other is a function on a connected manifold or as a result a special generic function on a standard sphere. 
For a connected component containing a singular point of index $1$, we lift as the right figure. The image is a $2$-dimensional closed disc with $6$ points in its corner. This is also regarded as a cobordism of two special generic functions such that one of the function is a function on a connected manifold and that the other is a function on a manifold having two connected
 components. For cobordisms of special generic functions, see \cite{saeki3} for example. 

For the connected component of the inverse image of a closed interval
 containing no singular point, we lift the function as the product of a Morse function with just two singular points on the homotopy sphere appeared as the corresponding inverse image and the closed interval.
 
Last, to glue together on each connected component of each point of the boudary of each closed interval in the target manifold, we apply Proposition \ref{prop:2} and Proposition \ref{prop:3}. By applying only Proposition \ref{prop:2}, we obtain a special generic map and this is also done in \cite{kitazawa}. By applying not only
 Proposition \ref{prop:2}, but also Proposition \ref{prop:3}, we can not always obtain a pseudo special generic map which can be realized as the quotient map of a special generic map; the Reeb space may be non-orientable. However, by
 the construction or the gluing, we can obtain the bundle of the first condition of Definition \ref{def:2} as an orientable bundle and as a fundamental discussion on the topological properties of linear bundles, the bundle is trivial. We can construct the second bundle as a trivial bundle. Note that desired smooth map $g$ from the target space of the pseudo special generic map into $\mathbb{R}$ is also obtained as a global map; in the locall constructions, the map is of course constructed as a smooth map.   

The second part can be shown in a similar method. As a most important different point, we locally lift functions into a map of higher dimensional maps. 
For a connected component containing a singular point of index $0$, we lift so that the image is a $n$-dimensional closed disc whose corner is $S^{n-1}$. This is also regarded as a cobordism of special generic maps whose images are closed discs such that one of the map is null and the other is a map on a connected manifold or as a result regarded as the canonical projection of a unit sphere. 
For a connected component containing a singular point of index $1$, we lift so that image is a $n$-dimensional closed disc whose corner is three disjoint ($n-1$)-dimensional stadard spheres. This is also regarded as a cobordism of two special generic maps such that one of the map is regarded as the canonical projection of a unit sphere and that the other is a disjoint union of two canonical projections of unit spheres. For cobordisms of special generic functions and maps, see \cite{saeki3} and also \cite{sadykov} for example. 
For the connected component of the inverse image of a closed interval
 containing no singular point, we lift a function naturally as the higher dimensional version of the case above.

We can discuss the last part similarly. Note that the 2nd homology and cohomology groups of the source manifold whose coefficient ring is $\mathbb{Z}$ vanish from Proposition \ref{prop:1}. We can obtain the bundle of the first condition of Definition \ref{def:2} as an orientable bundle over a $2$- or $3$-dimensional closed manifold and as a fundamental discussion on the topological properties of linear bundles, the bundle is trivial (since the second and third Stiefel Whitney classes vanish). We can construct the second bundle as a trivial bundle. Note that desired smooth map $g$ from the target space of the pseudo special generic map into $\mathbb{R}$ is also obtained as a global map.

\end{proof}

\begin{figure}
\includegraphics[width=35mm]{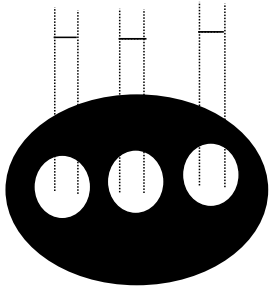}
\caption{An example of lifts of normal spherical Morse functions.}
\label{fig:3} 
\end{figure} 
\begin{figure}
\includegraphics[width=35mm]{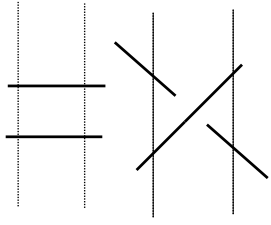}
\caption{Local topology of the surfaces near the black intervals of FIGURE \ref{fig:3}; the existence of the right case may make the target manifold of the pseudo special generic
 map non-orientable but in this case the bundles in the conditions of Definitions \ref{def:2} can be orientable.}
\label{fig:4}
\end{figure} 
 
For examples, see FIGURE \ref{fig:3} and FIGURE \ref{fig:4} ($n=2$ case).

Compare Theorem \ref{thm:1} and Theorem 5 of \cite{kitazawa}; we cannot construct a lift as a special generic map so that the underlying pseudo special generic map is trivial when the source manifold is non-orientable and the dimension of the Euclidean space of the target is larger $2$. 
Last, we have the following.

\begin{Cor}
\label{cor:1}
A spherical Morse function on a closed manifold of dimension larger than $2$ is represented as the composition of a smooth map into ${\mathbb{R}}^3$ regarded as the composition of a trivial pseudo special generic map and an appropriate smooth embedding of the target space of the map and the canonical projection ${\pi}_{3,1}$ defined by $(x_1,x_2,x_3) \in {\mathbb{R}}^3 \mapsto x_1 \in \mathbb{R}$.  
\end{Cor}
\begin{proof}
If the resulting pseudo special generic map is regarded as a map induced from a special generic map into the plane, then we can represent the Morse function as the composition of a special generic map into the plane obtained as a lift of the function and the canonical projection onto $\mathbb{R}$. We must sometimes push some local parts forward to realize the target space of the resulting pseudo special generic map in an appropriate Euclidean space or ${\mathbb{R}}^2 \times \mathbb{R}$. If in gluing local maps, we must apply the right case of FIGURE \ref{fig:4} (\ref{fig:3}), then we need
 to push this part forward to realize the target space of the resulting pseudo special generic map in an appropriate Euclidean space or ${\mathbb{R}}^2 \times \mathbb{R}$. From these discussions, we obtain the statement.   
\end{proof}
\begin{figure}
\includegraphics[width=55mm]{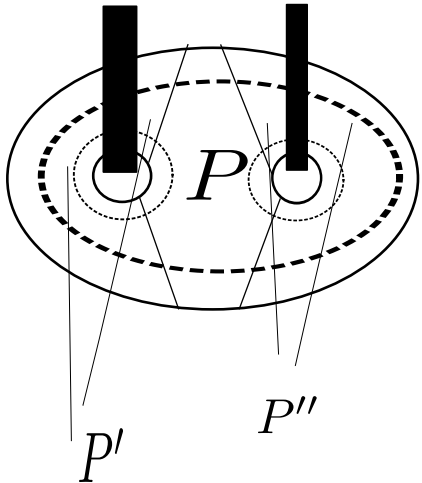}
\caption{The target manifold $W_p$ of a pseudo special generic map. 
$P^{\prime}$ and $P^{\prime \prime}$ are regions located in the left (right) side of the two thiin lines in the left (resp. right) and $P$ is the region of the center). Local surfaces
 in the black squares are like ones in FIGURE \ref{fig:4}.}
\label{fig:5}
\end{figure} 
\begin{Thm}
\label{thm:2}
A spherical Morse function $f$ on a closed manifold $M$ of dimension $m>2$ is represented as the composition of an embedding lift to ${\mathbb{R}}^n$
 where $n \geq \max \{\frac{3m+3}{2},m+2+1\}+1$ holds and the canonical projection ${\pi}_{n,1}$ defined by $(x_1,\cdots,x_n) \in {\mathbb{R}}^n \mapsto x_1 \in \mathbb{R}$ and we
can take the embedding lift so that the composition of the embedding and the canonical projection ${\pi}_{n,3}$ defined
 by $(x_1,\cdots,x_n) \in {\mathbb{R}}^n \mapsto (x_1,x_2,x_3) \in {\mathbb{R}}^3$ is a map as mentioned in Corollary \ref{cor:1}, represented as the composition of a trivial pseudo special generic map into the $2$-dimensional compact manifold and an embedding into ${\mathbb{R}}^3$.  
\end{Thm}
\begin{proof}
In the proof, we apply technique of constructions of embedding lifts of special generic
 maps such that the normal bundles of the singular sets are trivial used in \cite{nishioka} and for precise discussions, see also this. \\
We represent the Morse function $f$ as a composition as Corollary \ref{cor:1}; let $f_p$ be the pseudo
 special generic map into the $2$-dimensional compact manifold $W_p \subset {\mathbb{R}}^3$. We decompose this as presented in FIGURE \ref{fig:5}.

Let $C_p$ be a small collar neighborhood of $\partial W_p$. $C_p$ is, in the figure, the disjoint union of regions each of which is surrounded by a circle representing a connected component of the boundary $\partial W_p$ and a dotted circle. For the map ${f_p} {\mid}_{{f_p}^{-1}(C_p)}:{f_p}^{-1}(C_p) \rightarrow C_p \subset {\mathbb{R}}^3$, we
 can construct an embedding $e_C:{f_p}^{-1}(C_p) \rightarrow C_p \times {\mathbb{R}}^{n-3} \subset {\mathbb{R}}^{n}$ satisfying ${f_p} {\mid}_{{f_p}^{-1}(C_p)}={\pi}_{n,3} \circ e_C$. Since the pseudo special generic map is trivial, this lift can
 be extended to the whole space $W_p$ and to the outer space ${\mathbb{R}}^3$ by virtue of the fact that the space of all the smooth embeddings of $S^{m-2}$ into ${\mathbb{R}}^{n-3}$ with Whitney $C^{\infty}$ topology is simply connected or technique of \cite{nishioka}.
From these arguments, we have the desired embedding.
\end{proof}
\begin{Rem}
In \cite{kitazawa3}, the author has shown a statement similar to Theorem \ref{thm:2} for the case where we can lift the Morse function to a trivial pseudo special generic map represented also as a special generic map
 into ${\mathbb{R}}^2$ as Theorem 7  and in this, $n$
 is assumed to be not smaller than $\max \{\frac{3m+3}{2},m+2+1\}$. Note also that in the case, the source manifold is assumed to be orientable by the reason that Nishioka's
 technique \cite{nishioka} mentioned in the proof above is needed.  
\end{Rem}


\end{document}